\documentclass[11pt,letterpaper]{amsart}
\usepackage{preamble}
\counterwithin{equation}{section}

\begin{document}

\title[Dual Linear Programming Bounds for Sphere Packing]{Dual Linear Programming Bounds for Sphere Packing via Discrete Reductions}
\author[Rupert Li]{Rupert Li}
\date{\today}

\maketitle

\begin{abstract}
The Cohn-Elkies linear program for sphere packing, which was used to solve the 8 and 24 dimensional cases, is conjectured to not be sharp in any other dimension $d>2$.
By mapping feasible points of this infinite-dimensional linear program into a finite-dimensional problem via discrete reduction, we provide a general method to obtain dual bounds on the Cohn-Elkies linear program.
This reduces the number of variables to be finite, enabling computer optimization techniques to be applied.
Using this method, we prove that the Cohn-Elkies bound cannot come close to the best packing densities known in dimensions $3 \leq d \leq 13$ except for the solved case $d=8$.
In particular, our dual bounds show the Cohn-Elkies bound is unable to solve the 3, 4, and 5 dimensional sphere packing problems.
\end{abstract}

\section{Introduction}
The sphere packing problem is that of finding the densest packing of congruent balls in $\R^d$, i.e., the packing that covers the largest proportion of $\R^d$.
The case $d=1$ is trivial (consecutive line segments cover all of $\R$), and despite the optimal packings in 2 and 3 dimensions being intuitively natural, proving no packing can do better is quite challenging.
Thue \cite{thue1892om} proved the case $d=2$ in 1892, but the case $d=3$ was only proven in 2005 by Hales \cite{hales2005proof} with a computer-assisted proof that was formally verified in 2017 by a team of 22 people \cite{hales2017formal}.
These ad hoc proofs use the explicit structure of packings in $\R^2$ and $\R^3$, and do not seem to easily extend to higher dimensions.

A general upper bound technique for sphere packing was developed in 2003 by Cohn and Elkies \cite{cohn2003new}.
The Cohn-Elkies linear program can be viewed as the continuous analog of the Delsarte linear program in coding theory, as developed by Delsarte \cite{delsarte1972bounds} in 1972.
This is no coincidence: by viewing the sphere packing problem as that of picking the densest set of points (the centers of the balls) such that no pair of points are within a fixed distance from each other (twice the radius), we find that the sphere packing problem and the central problem of coding theory---picking codewords that are each at least some fixed distance apart---are analogous problems, just in different metric spaces.
The Cohn-Elkies linear programming bound yields the best upper bounds known for the sphere packing density in high dimensions \cite{sardari2020new}.
Feasible points of this infinite-dimensional linear program are functions $f:\R^d\to\R$; this yields a continuum of variables to optimize over, making it extremely difficult to find an optimal solution.
However, near-optimal points can be obtained by polynomials times a Gaussian for proper decaying, and these approximations suggest that the linear programming bound is sharp in $d=8$ and 24.

In a 2017 breakthrough, Viazovska \cite{viazovska2017sphere} found an optimal function for $d=8$ that matched the density of the $E_8$ root lattice, solving the 8-dimensional sphere packing problem.
Almost immediately afterwards, Cohn, Kumar, Miller, Radchenko, and Viazovska \cite{cohn2017sphere} similarly found an optimal function for $d=24$, solving the sphere packing problem for $d=24$.
To this day, these are the only two cases beyond $d=3$ that have been solved.

The Cohn-Elkies linear programming bound is conjectured to not be sharp in any other dimension $d>2$.
As the linear programming name suggests, the Cohn-Elkies linear program has a dual linear program, so proving non-sharpness, i.e., bounding below the Cohn-Elkies linear programming bound above the sphere packing density, amounts to finding a feasible solution to the dual linear program with objective value higher than the sphere packing density in the corresponding dimension.
The lack of a duality gap for this infinite-dimensional linear program is more subtle than for the finite-dimensional case, which follows from classical strong duality; the lack of a duality gap for the Cohn-Elkies linear program was conjectured by Cohn \cite{cohn2002new} and proven by Cohn, de Laat, and Salmon \cite{cohn2022three}.
Unlike the primal, however, it is much harder to obtain near-optimal feasible points of the dual.
This intractability stems from the fact that the dual amounts to optimizing over a space of measures, where it is believed that optimal measures will be supported on a discrete set of radii and thus are singular.
There is no known simple family of measures, analogous to the polynomials times Gaussians in the primal linear program, that yield good approximations to the dual linear program.
Cohn and Triantafillou \cite{cohn2022dual} optimize over spaces of modular forms to find feasible points of the dual linear program, yielding dual bounds for $d\in\{12,16,20,28,32\}$ that prove the Cohn-Elkies linear programming bound cannot come near the best known packing densities in these dimensions.
For dimensions $3 \leq d \leq 16$ except for the solved case $d=8$, upper bounds better than those from the Cohn-Elkies linear program have been proven, with the best known upper bounds coming from the recent work of Cohn, de Laat, and Salmon \cite{cohn2022three}.
In particular, the dual bounds for $d=12$ and 16 are greater than the known upper bound on the respective sphere packing density, proving non-sharpness of the Cohn-Elkies linear programming bound: it cannot equal the sphere packing density, and therefore cannot prove the 12 or 16-dimensional sphere packing problem.
However, for $d > 16$ we do not have better upper bounds than the Cohn-Elkies linear program, so $d=12$ and 16 are the only dimensions for which non-sharpness of the Cohn-Elkies linear programming bound is previously known.
De Courcy-Ireland, Dostert, and Viazovska \cite{de2022dual} are currently working to prove non-sharpness for $d=6$ using this modular form method.

We present a new method to create dual bounds for the Cohn-Elkies linear program.
The core of this method is a discrete reduction given by a composition of restrictions to map feasible points of the Cohn-Elkies linear program into a finite-dimensional problem, whose dual linear program is tractable to approximate.
Specifically, a feasible point of the Cohn-Elkies linear program is a function $f:\R^d\to\R$ satisfying certain constraints.
We restrict $f$ to $g:\Z^d\to\R$, and then restrict the resulting Fourier transform $\widehat g$ to $\widehat g_m: \Z_m^d\to\R$.
Due to the nature of modular forms, the method of Cohn and Triantafillou, which requires modular forms of weight $d/2$, is simplest when $d$ is even, since odd dimensions require modular forms of half-integral weight; their paper \cite{cohn2022dual} dealt only with $d$ a multiple of 4.
Our discrete reduction, however, is impartial to the modular congruences of $d$, yielding a more general method to create dual bounds, albeit a method which is limited by computational power as $d$ increases.
Using this method, we create dual bounds that prove the Cohn-Elkies linear programming bound is not sharp in dimensions $d=3$, 4, and 5, expanding the set of dimensions for which this is known to $\{3,4,5,12,16\}$.
For $6 \leq d \leq 13$ except for the sharp case $d=8$, we have the weaker result that the Cohn-Elkies linear programming bound is greater than the best known sphere packing densities in these dimensions.

Even when the Cohn-Elkies linear programming bound is not sharp, determining its optimal value is still of interest in its own right.
It can be interpreted as describing a sign uncertainty principle for a function and its Fourier transform, similar in form to the mathematical statement underlying the Heisenberg uncertainty principle; see Cohn and Gon\c{c}alves \cite{cohn2019optimal}.

In \cref{section: preliminaries}, we provide preliminary background on sphere packing and the Cohn-Elkies linear program.
We introduce our discrete reduction and the resulting discrete Cohn-Elkies linear program in \cref{section: Discrete LP}.
Our central result, \cref{theorem: discrete reduction}, demonstrates our dual bound method by proving that the discrete Cohn-Elkies linear programming bound is a lower bound of the (continuous) Cohn-Elkies linear programming bound.
In \cref{section: discrete dual LP}, we explicitly construct the discrete dual linear program and discuss how a feasible point of this dual creates a dual bound for the Cohn-Elkies linear programming bound via a summation formula.
We describe the algorithm used to find and verify discrete dual solutions in \cref{section: constructing dual bounds}.
Applying this algorithm, which relies on \cref{theorem: discrete reduction}, we prove in \cref{section: d = 3,section: d = 4 and 5} that the Cohn-Elkies linear programming bound is not sharp in dimensions 3, 4, and 5.
In \cref{section: beat construction}, we similarly find dual solutions that prove the Cohn-Elkies linear programming bound is greater than the best known sphere packing densities in dimensions 6 through 12, inclusive, except 8.
In \cref{section: general dual bound odd dimension}, we provide a general dual solution for odd dimensions $d \geq 5$, extending the results from the previous section to $d=13$.
We conclude in \cref{section: conclusion} with some computational observations and open questions.

\section{Sphere packing and the Cohn-Elkies linear program}\label{section: preliminaries}
We first introduce basic terminology regarding sphere packing, as well as the Cohn-Elkies linear programming bound.
Denoting the open ball of radius $r$ centered at $x\in\R^d$ as $B_r^d(x)$, a \emph{sphere packing} $\mc{P}$ in $\R^d$ is a disjoint union $\bigcup_{x\in C} B_r^d(x)$ of balls of a fixed radius $r$.
The \emph{upper density} $\Delta_{\mc{P}}$ of $\mc{P}$ is
\[ \Delta_{\mc{P}} = \limsup_{R\to\infty} \frac{\vol(B_R^d(0)\cap\mc{P})}{\vol(B_R^d(0))} = \limsup_{R\to\infty} \frac{\vol(B_R^d\cap\mc{P})}{\vol(B_R^d)}, \]
where $B_r^d$ denotes the radius $r$ ball in $\R^d$ centered at 0.
Standardizing $r=1$, if we wish to view the sphere packing problem as one of picking center-points rather than balls, the analogous quantity is the \emph{upper center density}
\[ \delta_\mc{P} = \limsup_{R\to\infty} \frac{|B_R^d\cap C|}{\vol(B_R^d)} = \frac{\Delta_\mc{P}}{\vol(B_1^d)}. \]
The \emph{sphere packing density} in $\R^d$ is then simply the highest density,
\[ \Delta_d = \sup_{\mc{P}} \Delta_\mc{P}, \]
the supremum over all sphere packings $\mc{P}\subset\R^d$.
We analogously define the \emph{sphere packing center density} $\delta_d = \frac{\Delta_d}{\vol(B_1^d)}$, which will typically have more elegant values.
While the sphere packing density $\Delta_d$ is invariant under scaling $\mc{P}$ and its corresponding $r$, the sphere packing center density $\delta_d$ is not, hence the normalization to $r=1$; the nominal value of $r$ does not change the essential nature of the sphere packing problem, and so when discussing the sphere packing center density we will always standardize $r=1$.

We normalize the \emph{Fourier transform} $\widehat f$ of $f:\R^d\to\R$ by
\[ \widehat f(y) = \int_{\R^d} f(x) e^{-2\pi i\left<x,y\right>}dx, \]
where $\left<x,y\right>=x_1y_1+\cdots+x_dy_d$ is the standard inner product on $\R^d$.
With these definitions, we can now formally state the Cohn-Elkies linear programming bound.
\begin{theorem}[Cohn and Elkies \protect{\cite{cohn2003new}}]\label{theorem: Cohn Elkies}
Let $f:\R^d\to\R$ be a continuous, integrable\footnote{
Cohn and Elkies \cite{cohn2003new} imposed stronger hypotheses on $f$, but using the approach from \cite[\S 4]{cohn2002new}, which shows smooth and rapidly decreasing functions can create bounds arbitrarily close to any bound from \cref{theorem: Cohn Elkies}, these hypotheses can be obviated.
This mollification of the hypotheses was first observed by Cohn and Kumar \cite{cohn2007universally}. }
function such that $\widehat f$ is also real-valued and integrable (i.e., $f$ is even).
Suppose the following constraints hold for some $r>0$:
\begin{enumerate}[label=(\arabic*)]
    \item $f(0),\widehat f(0)>0$,
    \item $f(x) \leq 0$ for $|x|\geq r$, and
    \item $\widehat f(y) \geq 0$ for all $y$.
\end{enumerate}
Then the sphere packing center density in $\R^d$ is at most
\[ \left(\frac{r}{2}\right)^d \frac{f(0)}{\widehat f(0)}. \]
\end{theorem}
This naturally yields an optimization problem to obtain the best upper bound.
Before we state this optimization problem, we will first make some simplifying assumptions about our feasible set of functions $f$.
As noted in the footnote, Cohn \cite{cohn2002new} showed that we may assume $f$ is smooth and rapidly decreasing without affecting the best bounds that can be obtained from \cref{theorem: Cohn Elkies}, where $f:\R^d\to\R$ is \emph{rapidly decreasing} if $f(x)=O\left(\frac{1}{(1+|x|)^k}\right)$ for all positive integers $k$, i.e., $f$ has faster-than-polynomial decay.
Notice that the constraints and objective value are invariant under scaling $f$, i.e., using $cf$ for $c>0$.
As a result, we can assume $\widehat f(0) \geq 1$ and remove the $\widehat f(0)$ term from the objective function (where any optimal $f$ will have $\widehat f(0) = 1$).
Furthermore, we may radially symmetrize $f$ by averaging over all of its rotations, so that $f(x)$ depends only on $|x|$, without changing the validity of the constraints nor the objective value.
Thus without loss of generality we may assume $f$ is radial.
This yields the following optimization problem.
\begin{problem}\label{problem: Cohn Elkies}
For a fixed $r>0$ and positive integer $d$, the \emph{Cohn-Elkies linear program} in $\R^d$ is
\begin{align*}
    \text{minimize} \quad & \left(\frac{r}{2}\right)^d f(0) \\
    \text{over all} \quad & \text{radial, smooth, rapidly decreasing functions } f:\R^d\to\R \\
    \text{such that} \quad & f(0) > 0 \\
    & \widehat f(0) \geq 1 \\
    & f(x) \leq 0 \quad \text{for all } |x| \geq r \\
    & \widehat f(y) \geq 0 \quad \text{for all } y.
\end{align*}
\end{problem}
The optimal objective value, i.e., the infimum of $\left(\frac{r}{2}\right)^d f(0)$, is the \emph{Cohn-Elkies linear programming bound} in $\R^d$.
By \cref{theorem: Cohn Elkies}, it is an upper bound on the sphere packing center density.
We call a feasible point $f$ an \emph{auxiliary function} on $\R^d$.

One may note that the Cohn-Elkies linear programming bound does not depend on $r$; in particular, by replacing $f$ with $x\mapsto c^d f(cx)$ for any $c>0$, whose Fourier transform is $y\mapsto\widehat f(y/c)$, and $r$ with $r/c$, we obtain a new function with the same objective value.
Hence we can without loss of generality assume $r=1$.
However, in our discrete reduction, the value of $r$ will matter, so we leave \cref{problem: Cohn Elkies} as is.

\section{The discrete Cohn-Elkies linear program}\label{section: Discrete LP}
We now introduce our discrete reduction.
For positive integer $m$, let $\Z_m$ denote $\Z/m\Z$, where we use coset representatives $x$ such that $-m/2 < x \leq m/2$.
Using these coset representatives, we may embed $\Z_m^d$ into $\Z^d\subset\R^d$, and thus inherit a metric from the Euclidean distance in $\R^d$.
We continue to use $|x|$ for $x\in\Z_m^d$ to denote the Euclidean distance from $x$ to 0, as well as $\left<x,y\right>$ to denote $x_1y_1+\cdots+x_dy_d$. 
Normalizing the \emph{discrete Fourier transform} $\widehat f$ of $f:\Z_m^d\to\R$ by
\[ \widehat f(y) = m^{-d/2} \sum_{x\in\Z_m^d} f(x) e^{-\frac{2\pi i}{m}\left<x,y\right>}, \]
we can define an analog of \cref{problem: Cohn Elkies} over functions on $\Z_m^d$.
\begin{problem}\label{problem: discrete Cohn Elkies}
For a fixed $r>0$ and positive integers $m$ and $d$, the \emph{discrete Cohn-Elkies linear program} in $\Z_m^d$ is
\begin{align*}
    \text{minimize} \quad & \left(\frac{r}{2}\right)^d f(0) \\
    \text{over all} \quad & \text{even functions } f:\Z_m^d\to\R \\
    \text{such that} \quad & \widehat f(0) \geq m^{-d/2} \\
    & f(x) \leq 0 \quad \text{for all } |x| \geq r \\
    & \widehat f(y) \geq 0 \quad \text{for all } y.
\end{align*}
\end{problem}
We analogously define the \emph{discrete Cohn-Elkies linear programming bound} in $\Z_m^d$ and \emph{auxiliary functions} on $\Z_m^d$.

Note that when $m=2$, this linear program coincides with the linear program for binary codes, i.e., the Delsarte linear program, where a Euclidean distance of $\sqrt{k}$ on $\{0,1\}^d$ is identified with a Hamming distance of $k$ on $\F_2^d$.
When $m>2$, however, our metric $|x|$ is unusual from the perspective of coding theory, e.g., it is neither the Hamming nor the Lee metric.
The discrete Cohn-Elkies linear program is still a linear programming bound for packing with the Euclidean metric in discrete space, and thus while it may be an unusual metric from a coding theory perspective, it is the natural metric from a discrete geometry perspective.

There are some slight differences between \cref{problem: Cohn Elkies,problem: discrete Cohn Elkies} that we now discuss.
Firstly, we could radially symmetrize $f:\R^d\to\R$.
However, over $\Z_m^d$ we may only average over the orbit of the symmetric group $S_d$ and changing the sign of the indices, i.e., average over $f(s\sigma x)$ for $\sigma\in S_d$ which permutes the indices of $x=(x_1,\dots,x_d)\in\Z_m^d$ and $s\in\Z_2^d\cong\{\pm1\}^d$ which changes the sign of the indices using coset representatives $-m/2<x\leq m/2$.
Let $G_d = \Z_2^d \rtimes S_d$, so that $G_d$ acts on $\Z_m^d$.
Thus we may assume $f$ is $G_d$-invariant, but cannot assume it is radial; this also implies $\widehat f$ is $G_d$-invariant.
For simplicity, we omit $G_d$-invariance from \cref{problem: discrete Cohn Elkies}, though we must still assume $f$ is even, so that $\widehat f$ is real-valued.
As $\Z_m^d$ is finite, we no longer have any concerns about continuity or decay, and thus do not need any analog of being smooth or rapidly decreasing.
Due to our normalization of the discrete Fourier transform, which includes the constant factor of $m^{-d/2}$ so that the discrete Fourier transform is its own inverse for even functions, the proper analog of $\widehat f(0) \geq 1$ in $\R^d$ is $\widehat f(0) \geq m^{-d/2}$ in $\Z_m^d$; we will see later that this yields the correct relative scaling of the objective functions of \cref{problem: Cohn Elkies,problem: discrete Cohn Elkies}.
Finally, as $\widehat f(0) \geq m^{-d/2} > 0$ and $\widehat f(y) \geq 0$ for all $y$, we always have
\[ f(0) = m^{-d/2} \sum_{y\in\Z_m^d} \widehat f(y) > 0, \]
so the corresponding constraint is no longer necessary.

A priori, there is no reason to believe \cref{problem: Cohn Elkies,problem: discrete Cohn Elkies} are related in any way beyond sharing the same form.
However, the following result, which provides a discrete reduction from an auxiliary function on $\R^d$ to an auxiliary function on $\Z_m^d$, reveals that these two problems are indeed related.

\begin{theorem}\label{theorem: discrete reduction}
For any $r>0$ and positive integers $d$ and $m\geq 2r$, if there exists an auxiliary function $f$ on $\R^d$ with objective value $\left(\frac{r}{2}\right)^d f(0)$, then there exists an auxiliary function $g_m$ on $\Z_m^d$ with the same $r$ and objective value $\left(\frac{r}{2}\right)^d g_m(0)$, where $g_m(0) \leq f(0)$.
\end{theorem}
\begin{proof}
Consider an auxiliary function $f$ on $\R^d$ with objective value $\left(\frac{r}{2}\right)^d f(0)$.
We will obtain $g_m$ by composing two maps $f\mapsto g \mapsto g_m$, where the intermediate step yields $g:\Z^d\to\R$ which is essentially an auxiliary function on $\Z^d$, although we have not explicitly defined what this means.
Both maps are obtained by domain restriction: $g$ is the restriction of $f$ from $\R^d$ to $\Z^d$, and the second map is obtained by restriction on the Fourier transform side, where $\widehat g_m$ is the restriction of $\widehat g$ from $T^d$ to $\Z_m^d$, where $T=\R/\Z$.
Here, for a function $h:\Z^d\to\C$, its Fourier transform is $\widehat h:T^d\to\C$ given by
\[ \widehat h(y) = \sum_{x\in\Z^d} h(x) e^{-2\pi i\left<x,y\right>}. \]
As $\Z^d$ and $T^d$ are Pontryagin duals, Fourier inversion also gives
\[ h(x) = \int_{T^d} \widehat h(y) e^{2\pi i\left<x,y\right>} dy; \]
this is essentially a multivariate Fourier series.

In order to have $g=f|_{\Z^d}$ and relate its Fourier transform back to $\widehat f$, we first define $\widehat{g}:T^d\to\R$ by
\[ \widehat g(y) = \sum_{n \in \Z^d} \widehat f(y+n), \]
using $\left(-\frac{1}{2},\frac{1}{2}\right]$ as coset representatives for $T$.
As $f$ and thus $\widehat f$ are smooth and rapidly decreasing, $\widehat g$ exists, i.e., the sum converges for all $y\in T^d$.
This also implies the inverse Fourier transform $g$ of $\widehat g$ exists, given by
\[ g(x) = \int_{T^d} e^{2\pi i\left<x,y\right>}\widehat g(y)\, dy = \int_{\R^d} e^{2\pi i\left<x,y\right>}\widehat f(y)\, dy = f(x) \]
for all $x \in \Z^d$.
In other words, we obtain $g$ by simply restricting $f$ to $\Z^d$, i.e., $g=f|_{\Z^d}$.
Then $\widehat g(0) \geq \widehat f(0) \geq 1$ and $\widehat g(y) \geq 0$ for all $y$.
We have $g(0)=f(0)>0$ and $g(x) = f(x) \leq 0$ for all $|x|\geq r$, using the metric inherited from $\R^d\supset\Z^d$.
By definition of the Fourier transform on $\Z^d$,
\[ \widehat g(0) = \sum_{x\in\Z^d} g(x), \]
which is absolutely convergent, i.e., $g\in L^1(\Z^d)$, as $f$ is rapidly decreasing.

Next, we similarly define $g_m:\Z_m^d\to\R$ by
\[ g_m(x) = \sum_{n \in m\Z^d} g(x+n). \]
As $g \in L^1(\Z^d)$,
\[ \sum_{x \in \Z_m^d} |g_m(x)| \leq \sum_{x \in \Z^d} |g(x)|, \]
which is finite, so $g_m \in L^1(\Z_m^d)$.
Then the Fourier transform of $g_m$ is given by
\[ \widehat g_m(y) = m^{-d/2} \sum_{x\in\Z_m^d} e^{-\frac{2\pi i}{m}\left<x,y\right>} g_m(x) = m^{-d/2} \sum_{x \in \Z^d} e^{-2\pi i\left<x,y/m\right>} g(x) = m^{-d/2} \widehat g(y/m) \]
for all $y \in \Z_m^d$.
Thus,
\[ \widehat g_m(0) = m^{-d/2} \widehat g(0) \geq m^{-d/2}\widehat f(0) \geq m^{-d/2}, \]
and $\widehat g_m(y) \geq 0$ for all $y$.
As usual, we use coset representatives in $\left(-m/2,m/2\right]$ for $\Z_m$ with $|x|$ for $x\in\Z_m^d$ as embedded in $\R^d$.
As $m \geq 2r$ by hypothesis in the theorem statement, if $|x| \geq r$ in $\Z_m^d$ then $|x+n|\geq r$ in $\Z^d$ for all $n\in m\Z^d$.
This implies that for $|x|\geq r$ in $\Z_m^d$,
\[ g_m(x) = \sum_{n\in m\Z^d} g(x+n) \leq 0. \]
As $|n|\geq m > r$ for all $n \in m\Z^d\setminus\{0\}$,
\[ g_m(0) = \sum_{n \in m\Z^d} g(n) \leq g(0) = f(0). \]
Evenness is preserved from $f\mapsto g \mapsto g_m$.
Hence, $g_m$ is an auxiliary function on $\Z_m^d$ satisfying $g_m(0) \leq f(0)$, completing the proof.
\end{proof}
As any auxiliary function $f$ on $\R^d$ can be mapped to an auxiliary function $g_m$ on $\Z_m^d$ while weakly decreasing the objective value, we immediately have the following corollary.
\begin{corollary}\label{corollary: discrete reduction}
For any $r>0$ and positive integers $d$ and $m \geq 2r$, the Cohn-Elkies linear programming bound in $\R^d$ is greater than or equal to the discrete Cohn-Elkies linear programming bound in $\Z_m^d$.
\end{corollary}

This provides a powerful method to create dual bounds on \cref{problem: Cohn Elkies}.
As \cref{problem: discrete Cohn Elkies} is a finite-dimensional linear program, i.e., it has finitely many variables, one for each value $f(x)$ for $x\in\Z_m^d$, linear programming duality immediately holds.
By weak duality, any feasible point of the dual of \cref{problem: discrete Cohn Elkies} provides a dual lower bound on the discrete Cohn-Elkies linear programming bound in $\Z_m^d$, which is a lower bound on the Cohn-Elkies linear programming bound in $\R^d$.
And by strong duality, these dual bounds on \cref{problem: discrete Cohn Elkies} are innately lossless, in the sense that because the duality gap is zero, the best dual bound equals the best primal objective, i.e., the discrete Cohn-Elkies linear programming bound.
Note that there is still a potentially nonzero gap between the discrete Cohn-Elkies linear programming bound and the Cohn-Elkies linear programming bound.

Intuitively, one should believe this gap goes to zero as $m,r\to\infty$.
Increasing $r$ is equivalent to stretching an auxiliary function $f$ on $\R^d$, so as $r\to\infty$ we see that the sampling $g=f|_{\Z^d}$ of $f$ at integer lattice points becomes increasingly fine, and in the limit ideally captures all information about $f$.
For the second step, $\widehat g_m$ samples $\widehat g$ at all points $\frac{y}{m}$ for $y\in\Z_m^d$, i.e., all points whose coordinates are multiples of $\frac{1}{m}$.
Then as $m\to\infty$, we see that the sampling $\widehat g_m$ of $\widehat g$ also becomes increasingly fine, and so in the limit ideally contains all information about $\widehat g$ or equivalently $g$.
Thus as $m,r\to\infty$, we expect $g_m$ to capture essentially all information about $f$.

Thus, for suitably large values of $m$ and $r$, we may find dual bounds on \cref{problem: discrete Cohn Elkies} that are fairly close to the Cohn-Elkies linear programming bound.
In particular, finding a dual bound above the sphere packing center density, for any value of $m$ and $r$, would imply the Cohn-Elkies linear programming bound is not sharp.
As the Cohn-Elkies linear program is not believed to be sharp in any dimension $d>2$ other than $d\in\{8,24\}$, there should be a gap between the sphere packing center density and the Cohn-Elkies linear programming bound.
If this is true, then in any of these dimensions, suitable values of $m$ and $r$ would yield a dual bound that proves the Cohn-Elkies linear programming bound is not sharp.

\section{The discrete dual Cohn-Elkies linear program}\label{section: discrete dual LP}
We now explicitly discuss how a feasible dual point of \cref{problem: discrete Cohn Elkies} creates a dual bound on \cref{problem: Cohn Elkies}.
For simplicity, in order to have nonnegativity constraints on all of our primal variables, we will consider \cref{problem: discrete Cohn Elkies} as a linear program over the variables $\widehat f(y)$ rather than $f(x)$.
Then we have constraints for all $|x|\geq r$, each of which yields a dual variable $\mu(x)$.
Letting $\widetilde\lambda$ (we will reserve $\lambda$ for a more important variable later) be the dual variable corresponding to the constraint $-\widehat f(0) \leq -m^{-d/2}$, the dual of \cref{problem: discrete Cohn Elkies} is the following linear program:
\begin{align*}
    \text{maximize} \quad & m^{-d/2}\widetilde\lambda \\
    \text{over all} \quad & \widetilde\lambda \geq 0, \mu(x) \geq 0 \text{ for each } |x| \geq r \\
    \text{such that} \quad & \sum_{|x|\geq r} \mu(x)f(x) - \widetilde\lambda\widehat f(0) + \left(\frac{r}{2}\right)^d f(0) \text{ has an expansion of the form } \sum_{y\in\Z_m^d} \lambda(y)\widehat f(y) \\
    & \text{ with } \lambda(y) \geq 0 \text{ for all } y.
\end{align*}
Defining $\mu(x) = 0$ for $0<|x|<r$ and $\mu(0)=\left(\frac{r}{2}\right)^d$, we have $\lambda$ is the Fourier transform of $\mu$.
Then
\begin{align}\label{eq: discrete dual summation formula}
    \sum_{x\in\Z_m^d} \mu(x) f(x) = \sum_{y \in \Z_m^d} \lambda(y) \widehat f(y)
\end{align}
is a summation formula that holds for any $f:\Z_m^d\to\R$, where $\mu(x),\lambda(y)\geq0$ for all $x,y$, and $\lambda(0) \geq \widetilde\lambda$.
Any optimal solution will have $\widetilde\lambda = \lambda(0)$, so we may restate the dual of \cref{problem: discrete Cohn Elkies} as follows.
\begin{problem}\label{problem: discrete dual Cohn Elkies}
For a fixed $r>0$ and positive integers $m$ and $d$, the \emph{discrete dual Cohn-Elkies linear program} in $\Z_m^d$ is
\begin{align*}
    \text{maximize} \quad & m^{-d/2}\lambda(0) \\
    \text{over all} \quad & \text{even functions }\mu,\lambda : \Z_m^d\to\R \\
    \text{such that} \quad & \lambda(y) \geq 0 \quad \text{for all } y \\
    & \lambda = \widehat\mu \\
    & \mu(x) \geq 0 \quad \text{for all } |x|\geq r \\
    & \mu(x) = 0 \quad \text{for all } 0 < |x| < r \\
    & \mu(0) = \left(\frac{r}{2}\right)^d.
\end{align*}
\end{problem}
We call a feasible point $\mu$, which uniquely determines $\lambda$, a \emph{dual auxiliary function} on $\Z_m^d$.
Notice that as we may assume an auxiliary function on $\Z_m^d$ is $G_d$-invariant in \cref{problem: discrete Cohn Elkies}, we may also assume $\mu$ and $\lambda$ are $G_d$-invariant in \cref{problem: discrete dual Cohn Elkies}.
Any feasible point of \cref{problem: discrete dual Cohn Elkies} yields a summation formula by \cref{eq: discrete dual summation formula}; using the notation of \cref{theorem: discrete reduction}, this implies that any auxiliary function $f$ on $\R^d$ has objective value $\left(\frac{r}{2}\right)^d f(0)$ bounded below by
\begin{align}
    \left(\frac{r}{2}\right)^d f(0)
    \geq \left(\frac{r}{2}\right)^d g_m(0)
    \geq \sum_{x\in\Z_m^d} \mu(x) g_m(x)
    = \sum_{y\in\Z_m^d} \lambda(y)\widehat g_m(y)
    \geq \lambda(0)\widehat g_m(0)
    \geq m^{-d/2} \lambda(0).
\end{align}
This provides a new method to create dual bounds on the Cohn-Elkies linear program, with the potential to prove that the Cohn-Elkies linear programming bound is not sharp.

\section{Constructing dual bounds}\label{section: constructing dual bounds}
We now describe the algorithm used to find and verify a dual auxiliary function $\mu$ on $\Z_m^d$.
If $r^2$ is not an integer, then increasing $r$ to $\sqrt{\ceil{r^2}}$ increases the objective function of the discrete Cohn-Elkies linear program without changing its constraints, as $|x|^2$ is always an integer.
To get the best dual bounds on the Cohn-Elkies linear program, we may thus assume $r^2$ is an integer where $1 \leq r \leq \frac{m}{2}$, so that \cref{theorem: discrete reduction} holds.
For any values $d$, $m$, and $r$, we can approximately solve for the discrete Cohn-Elkies linear programming bound via a linear programming solver.
Once suitable values of $d$, $m$, and $r$ (i.e., parameters that yield large optimal values) have been found, we can find a rational dual auxiliary function $\mu$ on $\Z_m^d$, i.e., $\mu(x)\in\Q$ for all $x\neq0$, although $\mu(0)=\left(\frac{r}{2}\right)^d$ may be constrained to be irrational.
This allows us to calculate $\lambda=\widehat\mu$ exactly, using exact rational arithmetic and symbolic expressions.
Using real interval arithmetic, we can verify by computer that $\lambda(y) \geq 0$ for all $y$.
Then $\lambda(0)\in\Q(r)$ where $r^2\in\N$, which yields an exact feasible objective value $m^{-d/2}\lambda(0)\in\Q(r,\sqrt{m})$ to \cref{problem: discrete dual Cohn Elkies}.
This is a dual lower bound to the Cohn-Elkies linear programming bound in $\R^d$, whose decimal expansion can be rounded down for convenience.

This verification process requires a rational dual auxiliary function $\mu$; however, the computer optimization techniques will only provide a floating point solution for $\mu$.
Rounding this approximate solution to rationals is a natural approach, but it has one obstacle that needs to be addressed.
In particular, by complementary slackness conditions we expect $\lambda(y)=0$ for various values of $y$; it is not guaranteed that rounding the approximate solution for $\mu$ to rationals will preserve nonnegativity of these $\lambda(y)$.
And rounding $\lambda$ to rationals is almost certainly a worse idea, as ensuring $\mu(x)=0$ for all $0<|x|<r$ and $\mu(0)=\left(\frac{r}{2}\right)^d$ is even harder.
However, constraining $\lambda(y)\geq\varepsilon$ for some fixed $\varepsilon>0$ is a simple yet surprisingly effective way to obviate this issue.
This stronger constraint would naturally decrease the optimal value, but in practice, a value like $\varepsilon=10^{-10}$ is a sufficiently large buffer to ensure rounding $\mu$ to rationals keeps $\lambda(y)\geq0$ for all $y$, without impacting the resulting dual bound significantly.

As noted previously, we may assume $\mu$ and $\lambda$ are $G_d$-invariant functions on $\Z_m^d$.
Then the values that $\mu$ and $\lambda$ take on for coset representatives of $\Z_m^d/G_d$ contain sufficient information to recover $\mu$ and $\lambda$.
Thus, for computational efficiency, it is useful to reduce auxiliary functions $f$ on $\Z_m^d$ as vectors indexed by $\Z_m^d/G_d$.
Let $\sim$ be the equivalence relation defined by the action of $G_d$, i.e., for $x,y\in\Z_m^d$, we have $x\sim y$ if $x = \sigma y$ for some $\sigma\in G_d$.
Then viewing $f$ as a vector indexed over $\Z_m^d/G_d$, we find $\widehat f = \mc{F} f$, where the $x,y$ entry of $\mc{F}$ for $x,y\in\Z_m^d/G_d$ is
\[ \mc{F}_{xy} = m^{-d/2} \sum_{z \sim y} e^{-\frac{2\pi i}{m}\left<x,z\right>}. \]
The matrix $\mc{F}$ represents the $G_d$-symmetrized discrete Fourier transform, and satisfies $\mc{F}^{-1}=\mc{F}$.
In order to write the summation formula, \cref{eq: discrete dual summation formula}, as a sum over $\Z_m^d/G_d$, i.e., in the form $\mu^T f = \lambda^T \widehat f$ where $\mu$ and $\lambda$ are vectors indexed over $\Z_m^d/G_d$, we must scale each entry $\mu_x$ and $\lambda_x$ by the size of the orbit of $x$.
In other words, for all $x\in\Z_m^d/G_d$,
\[ \mu_x = \mu(x)\cdot|\{y\in\Z_m^d|y\sim x\}| = \mu(x) \cdot |\Orb(x)|,\]
and similarly for $\lambda_x$, where we use $\Orb(x)=\{y\in\Z_m^d\,|\,y\sim x\}$ to denote the orbit of $x$ under $G_d$.
From now on, we will use the vector interpretation of $\mu$ and $\lambda$, i.e., with the rescaling by orbit size.
Then $\lambda = \mc{F}^T \mu$.
Notice that 0 is in its own orbit, so this rescaling does not affect $\mu_0$ or $\lambda_0$.
For clarity, we restate \cref{problem: discrete dual Cohn Elkies} using this vector notation.
\begin{problem}\label{problem: radial discrete dual Cohn Elkies}
For a fixed $r>0$ and positive integers $m$ and $d$, the (radialized) discrete dual Cohn-Elkies linear program in $\Z_m^d$ is
\begin{align*}
    \text{maximize} \quad & m^{-d/2}\lambda_0 \\
    \text{over all} \quad & \text{real vectors }\mu,\lambda \text{ indexed by } \Z_m^d/G_d \\
    \text{such that} \quad & \lambda_y \geq 0 \quad \text{for all } y \\
    & \lambda = \mc{F}^T\mu \\
    & \mu_x \geq 0 \quad \text{for all } |x|\geq r \\
    & \mu_x = 0 \quad \text{for all } 0 < |x| < r \\
    & \mu_0 = \left(\frac{r}{2}\right)^d.
\end{align*}
\end{problem}
We will provide dual bounds via feasible points to \cref{problem: radial discrete dual Cohn Elkies}.

\section{Dual bounds in dimension 3}\label{section: d = 3}
Hales \cite{hales2005proof} proved the sphere packing center density for $d=3$ is $\delta_3 = 2^{-5/2} < 0.17677670$.
We provide a dual bound on the Cohn-Elkies linear program that proves it is not sharp, i.e., that the Cohn-Elkies linear programming bound is strictly larger than $\delta_3$.
This implies the Cohn-Elkies linear program is unable to prove the 3-dimensional sphere packing problem.
\begin{theorem}\label{theorem: d=3}
The 3-dimensional Cohn-Elkies linear programming bound is greater than $0.18398089$.
\end{theorem}
\begin{proof}
Using $m=53$ and $r=\sqrt{89}$, we find a feasible point to \cref{problem: radial discrete dual Cohn Elkies} whose objective value, truncated (i.e., rounded down) to eight decimal points, is $0.18398089$.
The exact rational dual auxiliary function $\mu$, as well as the objective value $m^{-d/2}\lambda_0$, are too large to include in the paper (e.g., the exact expression for the objective value involves rationals of more than 7000 digits).
See supplementary files, available on the arXiv.org e-print archive for this paper, for $\mu$ and its objective value.
\end{proof}

Each supplementary solution file includes a list of representatives of $\Z_m^d/G_d$, the exact dual auxiliary function $\mu$ indexed by these representatives, and the objective value $m^{-d/2}\lambda_0$, i.e., the dual bound that $\mu$ creates.
Code to verify that $\mu$, which uniquely determines $\lambda=\mc{F}^T\mu$, satisfies the constraints of \cref{problem: radial discrete dual Cohn Elkies} is included in a separate supplementary file.

The lowest value of $m$ for which the discrete dual Cohn-Elkies linear program in $\Z_m^3$ yields a dual bound higher than $\delta_3$ is $m=21$, using $r=\sqrt{41}$.

We remark that the best upper bound on $\delta_3$ not using the computer-assisted proof of Hales is 0.183889, due to Cohn, de Laat, and Salmon \cite{cohn2022three}.
Thus, even without Hales' solution to the 3-dimensional sphere packing problem, \cref{theorem: d=3} would still be sufficient to imply non-sharpness of the Cohn-Elkies linear programming bound in dimension 3.
While this observation is not needed for $d=3$, in dimensions where the sphere packing problem is unresolved, this approach will be necessary to prove non-sharpness, as done in the following section for $d=4$ and 5.

\section{Dual bounds in dimensions 4 and 5}\label{section: d = 4 and 5}
Unlike the 3-dimensional case, the sphere packing problem is currently unsolved in dimensions 4 and 5.
The $D_4$ root lattice packing \cite{korkine1872formes}, which is conjectured to be optimal, has upper center density $\frac{1}{8}$.
The best known upper bound for $\delta_4$ is due to Cohn, de Laat, and Salmon \cite{cohn2022three}, and is 0.12891, rounded up.
Thus $0.125 \leq \delta_4 \leq 0.12891$, where it is believed that $\delta_4$ equals its lower bound of $\frac{1}{8}$.
Despite the upper bound on $\delta_4$ not matching the lower bound, we are still able to provide a dual bound on the Cohn-Elkies linear program that exceeds the upper bound, proving non-sharpness and thus that the Cohn-Elkies linear program is unable to prove the 4-dimensional sphere packing problem.

\begin{theorem}\label{theorem: d=4}
The 4-dimensional Cohn-Elkies linear programming bound is greater than $0.12914461$.
\end{theorem}
\begin{proof}
Using $m=30$ and $r=\sqrt{66}$, we find a feasible point to \cref{problem: radial discrete dual Cohn Elkies} whose objective value truncated to eight decimal points is $0.12914461$.
See supplementary files for the exact rational dual auxiliary function $\mu$ and the exact objective value $m^{-d/2}\lambda_0$.
\end{proof}

The lowest value of $m$ for which the discrete dual Cohn-Elkies linear program in $\Z_m^4$ yields a dual bound higher than the upper bound on $\delta_4$ is $m=30$.
However, the lowest value of $m$ which yields a dual bound higher than the lower bound on $\delta_4$ is $m=16$, using $r=\sqrt{30}$.

For $d=5$, the best known sphere packing is the $D_5$ root lattice \cite{korkine1873formes} with an upper center density of $2^{-7/2}$.
The best known upper bound for $\delta_5$ is due to Cohn, de Laat, and Salmon \cite{cohn2022three}, and is 0.09740, rounded up.
We provide a dual bound on the Cohn-Elkies linear program that exceeds this upper bound.
\begin{theorem}\label{theorem: d=5}
The 5-dimensional Cohn-Elkies linear programming bound is greater than $0.09826308$.
\end{theorem}
\begin{proof}
Using $m=24$ and $r=\sqrt{34}$, we find a feasible point to \cref{problem: radial discrete dual Cohn Elkies} whose objective value truncated to eight decimal points is $0.09826308$.
See supplementary files for the exact rational dual auxiliary function $\mu$ and the exact objective value $m^{-d/2}\lambda_0$.
\end{proof}
Thus, the Cohn-Elkies linear program is not sharp in dimension 5, i.e., cannot prove the 5-dimensional sphere packing problem.
The lowest value of $m$ for which the discrete dual Cohn-Elkies linear program in $\Z_m^5$ proves non-sharpness is $m=24$.
However, the lowest value of $m$ which yields a dual bound higher than the $D_5$ root lattice packing's upper center density is $m=10$, using $r=\sqrt{14}$.

\section{Dual bounds for $6 \leq d \leq 12$ except $d=8$}\label{section: beat construction}
Except for the solved case $d=24$, for $17 \leq d \leq 179$, the best known upper bound \cite{afkhami2020high} is essentially the same as the best known bound from the Cohn-Elkies linear program.
The bounds from \cite{cohn2022three} are conjecturally better, but have not been explicitly computed for $d \geq 17$.
It is thus nearly impossible to prove non-sharpness of the Cohn-Elkies linear programming bound, as we did for $3\leq d \leq 5$, without a better upper bound that does not come from the Cohn-Elkies linear program itself.
As such, the next best result towards non-sharpness is proving a dual bound that exceeds the upper center density of the best known sphere packing.
For the small values of $d$ that we consider in this paper, the best known sphere packings are conjecturally believed to be optimal.

For $6 \leq d \leq 12$ except for the solved case $d=8$, while there exist upper bounds better than those provided by the Cohn-Elkies linear program, the strongest of which is by Cohn, de Laat, and Salmon \cite{cohn2022three}, due to increasing computational cost we were unable to prove non-sharpness of the Cohn-Elkies linear programming bound in these cases.
However, in these cases we were still able to prove the next best alternative, a dual bound exceeding the upper center density of the best known sphere packing.

\subsection{Dimension 6}\label{subsection: d = 6}
For $d=6$, the best known sphere packing is the $E_6$ root lattice \cite{korkine1873formes} with upper center density $2^{-3}3^{-1/2} < 0.07216879$.
We find a feasible point to \cref{problem: radial discrete dual Cohn Elkies} whose objective value exceeds this upper center density.
\begin{theorem}\label{theorem: d=6}
The 6-dimensional Cohn-Elkies linear programming bound is greater than $0.07632412$.
\end{theorem}
\begin{proof}
See supplementary files for an exact rational dual auxiliary function $\mu$ for $m=12$ and $r=4$ whose exact objective value $m^{-d/2}\lambda_0$ rounds down to $0.07632412$.
\end{proof}
Hence, the Cohn-Elkies linear program is unable to prove that the $E_6$ root lattice is optimal in dimension 6.
The lowest value of $m$ which yields a dual bound higher than the $E_6$ root lattice packing's upper center density is $m=8$, using $r=4$.

The best known upper bound on $\delta_6$ is $0.07939747$, rounded up \cite{cohn2022three}.
The ongoing work of de Courcy-Ireland, Dostert, and Viazovska \cite{de2022dual} to prove non-sharpness in $d=6$ would create a dual bound exceeding this best known upper bound, and hence would be a stronger result than \cref{theorem: d=6}.
Their provisional value for the dual bound is 0.0795223.

\subsection{Dimension 7}\label{subsection: d = 7}
For $d=7$, the best known sphere packing is the $E_7$ root lattice \cite{korkine1873formes} with upper center density $2^{-4}=0.0625$.
As before, we find a feasible point to \cref{problem: radial discrete dual Cohn Elkies} whose objective value exceeds this upper center density.
\begin{theorem}\label{theorem: d=7}
The 7-dimensional Cohn-Elkies linear programming bound is greater than $0.06374745$.
\end{theorem}
\begin{proof}
See supplementary files for an exact rational dual auxiliary function $\mu$ for $m=8$ and $r=4$ whose exact objective value $m^{-d/2}\lambda_0$ rounds down to $0.06374745$.
\end{proof}
The Cohn-Elkies linear program is therefore unable to prove that the $E_7$ root lattice is optimal in dimension 7.
The provided solution with $m=8$ is the lowest value of $m$ that yields a dual bound higher than the $E_7$ root lattice packing's upper center density.

The best known upper bound on $\delta_7$ is $0.06797101$, rounded up \cite{cohn2022three}.

\subsection{Dimension 9}\label{subsection: d = 9}
The 8-dimensional sphere packing problem was solved via the Cohn-Elkies linear programming bound, so the bound is known to be sharp.
Hence, the next dimension of interest is $d=9$, for which the best known packing is the $\Lambda_9$ packing \cite{chaundy1946arithmetic} with upper center density $2^{-9/2} < 0.04419418$.
As before, we find a dual bound exceeding this upper center density.
\begin{theorem}\label{theorem: d=9}
The 9-dimensional Cohn-Elkies linear programming bound is greater than or equal to $\frac{1}{20}=0.05$.
\end{theorem}
\begin{proof}
See supplementary files for an exact rational dual auxiliary function $\mu$ for $m=4$ and $r=2$ whose exact objective value $m^{-d/2}\lambda_0=4^{-9/2}\frac{2^7}{5}=\frac{1}{20}=0.05$.
The approximate solution for $\mu$ from computer optimization seemed to take on values extremely close to $\frac{1}{5}\Z$, so rather than using the buffering technique of constraining $\lambda(y) \geq \varepsilon > 0$ from \cref{section: constructing dual bounds}, we simply rounded the approximate solution to take values in $\frac{1}{5}\Z$.
This simple solution was then exactly verified, as before.
The solution is elegant enough to explicitly state $\lambda$ (though $\mu$ has a larger support, so we do not include its value in the paper): we have $\lambda_x = \frac{128}{5}$ for $x=0=(0,\dots,0)\in\Z_4^9/G_9$ and $x=(2,\dots,2)$; $\lambda_x = \frac{1152}{5}$ for representatives $x=(0,0,0,0,2,2,2,2,2)$ and $(0,0,0,0,0,2,2,2,2)$; and $\lambda_x=0$ for all other representatives $x \in \Z_4^9/G_9$.
\end{proof}
The provided solution with $m=4$ is the lowest value of $m$ that yields a dual bound higher than the best known packing's upper center density.

The best known upper bound on $\delta_9$ is 0.05794146, rounded up \cite{cohn2022three}.

\subsection{Dimension 10}\label{subsection: d = 10}
For $d=10$, the best known sphere packing \cite{best1980binary} has upper center density $5\cdot 2^{-7} = 0.0390625$.
We provide a dual bound exceeding this upper center density.
\begin{theorem}\label{theorem: d=10}
The 10-dimensional Cohn-Elkies linear programming bound is greater than or equal to $\frac{1}{24}> 0.04166666$.
\end{theorem}
\begin{proof}
See supplementary files for an exact rational dual auxiliary function $\mu$ for $m=4$ and $r=2$ whose exact objective value $m^{-d/2}\lambda_0=4^{-10/2}\frac{2^7}{3}=\frac{1}{24}$.
The approximate solution for $\mu$ from computer optimization takes on values extremely close to $\frac{1}{3}\Z$, so rather than using the buffering technique from \cref{section: constructing dual bounds}, we rounded the approximate solution to take values in $\frac{1}{3}\Z$.
This simple solution was then exactly verified, as before.
The solution is elegant enough to explicitly state $\lambda$: we have $\lambda_x = \frac{128}{3}$ for $x=0=(0,\dots,0)\in\Z_4^{10}/G_{10}$ and $x=(2,\dots,2)$; $\lambda_x = \frac{2560}{3}$ for $x=(0,0,0,0,1,1,2,2,2,2)$; $\lambda_x=\frac{256}{3}$ for $x=(0,0,0,0,0,2,2,2,2,2)$; and $\lambda_x=0$ for all other representatives $x \in \Z_4^{10}/G_{10}$.
\end{proof}
The provided solution with $m=4$ is the lowest value of $m$ that yields a dual bound higher than the best known upper center density.

The best known upper bound on $\delta_{10}$ is $0.05623564$, rounded up \cite{cohn2022three}.

\subsection{Dimension 11}\label{subsection: d = 11}
For $d=11$, the best known sphere packing \cite{leech1971sphere} has upper center density $3^2\cdot2^{-8}=\frac{9}{256}=0.03515625$.
We provide a dual bound exceeding this upper center density.
\begin{theorem}\label{theorem: d=11}
The 11-dimensional Cohn-Elkies linear programming bound is greater than or equal to $\frac{1}{24}> 0.04166666$.
\end{theorem}
\begin{proof}
See supplementary files for an exact rational dual auxiliary function $\mu$ for $m=4$ and $r=2$ whose exact objective value $m^{-d/2}\lambda_0 = 4^{-11/2}\frac{2^8}{3}=\frac{1}{24}$.
Similar to the case $d=10$, the approximate solution for $\mu$ consists of values extremely close to $\frac{1}{3}\Z$, so we rounded the approximate solution to lie in $\frac{1}{3}\Z$.
This simple solution was then exactly verified, as before.
The solution is elegant enough to explicitly state $\lambda$: we have $\lambda_x = \frac{256}{3}$ for $x=0=(0,\dots,0)\in\Z_4^{11}/G_{11}$ and $x=(2,\dots,2)$; $\lambda_x = \frac{1280}{3}$ for $x=(0,0,0,0,0,2,2,2,2,2,2)$ and $(0,0,0,0,0,0,2,2,2,2,2)$; $\lambda_x=1024$ for $x=(0,0,0,0,0,1,2,2,2,2,2)$; and $\lambda_x=0$ for all other representatives $x \in \Z_4^{11}/G_{11}$.
\end{proof}
The provided solution with $m=4$ is the lowest value of $m$ that yields a dual bound higher than the best known upper center density.

The best known upper bound on $\delta_{11}$ is $0.05664513$, rounded up \cite{cohn2022three}.

\subsection{Dimension 12}\label{subsection: d = 12}
For $d=12$, the best known sphere packing \cite{coxeter1953extreme} has upper center density $3^{-3}=\frac{1}{27} < 0.03703704$.
This is the first dimension for which the best known sphere packing upper center density increased compared to the previous value of $d$.
We provide a dual bound exceeding this upper center density.
\begin{theorem}\label{theorem: d=12}
The 12-dimensional Cohn-Elkies linear programming bound is greater than or equal to $\frac{1}{24}> 0.04166666$.
\end{theorem}
\begin{proof}
See supplementary files for an exact rational dual auxiliary function $\mu$ for $m=4$ and $r=2$ whose exact objective value $m^{-d/2}\lambda_0=4^{-12/2}\frac{2^9}{3}=\frac{1}{24}$.
Similar to the cases $d=10$ and 11, the approximate solution for $\mu$ consists of values extremely close to $\frac{1}{3}\Z$, so we rounded the approximate solution to lie in $\frac{1}{3}\Z$.
This simple solution was then exactly verified, as before.
The solution is elegant enough to explicitly state $\lambda$: we have $\lambda_x = \frac{512}{3}$ for $x=0=(0,\dots,0)\in\Z_4^{12}/G_{12}$ and $x=(2,\dots,2)$; $\lambda_x = \frac{11264}{3}$ for $x=(0,0,0,0,0,0,2,2,2,2,2,2)$; and $\lambda_x=0$ for all other representatives $x\in\Z_4^{12}/G_{12}$.
\end{proof}
The provided solution with $m=4$ is the lowest value of $m$ that yields a dual bound higher than the best known upper center density.

The best known upper bound on $\delta_{12}$ is 0.05969745, rounded up \cite{cohn2022three}.
The dual bound of Cohn and Triantafillou \cite{cohn2022dual} exceeds this upper bound, proving non-sharpness for $d=12$.
Their result is thus stronger than \cref{theorem: d=12}.
This illustrates how, limited by computational power, the modular form method scales better than our discrete reduction method as dimension increases, so it can achieve better dual bounds in high dimensions.
On the other hand, the modular form method is harder to optimize in low dimensions, so the discrete reduction method is better in low dimensions: for example, the discrete reduction method is able to prove non-sharpness for $3\leq d \leq 5$, which no other method has done before.
Moreover, the discrete reduction method is more easily applied to arbitrary dimensions.

\section{General dual bounds in odd dimensions}\label{section: general dual bound odd dimension}
Inspired by the elegantly simple dual auxiliary functions $\lambda$ from the $9 \leq d \leq 12$ cases for $m=4$ and $r=2$, we identify a general dual auxiliary function for odd $d$.
For brevity of notation, to refer to an element $(0,\dots,0,1,\dots,1,2\dots,2)\in\Z_4^d$ where there are $a$ zeros, $b$ ones, and $c$ twos, we use the notation $(0^a,1^b,2^c)$, where we may omit terms, e.g., $(0^a,1^b)$, when there are no entries of that type, e.g., $c=0$.
Notice that the set of $(0^a,1^b,2^c)$ for all $a,b,c\geq0$ such that $a+b+c=d$ forms a complete set of coset representatives for $\Z_4^d/G_d$.
We continue to use the notation $0 \in \Z_m^d$ to refer to $(0,\dots,0)\in\Z_m^d$.
\begin{theorem}\label{theorem: general dual auxiliary odd dimension}
For positive integer $k \geq 2$, let $d=2k+1$.
Then $\lambda:\Z_4^d/G_d\to\R$ defined by
\begin{align*}
    \lambda_0 = \lambda_{(2^d)} &= \frac{2^{2k-1}}{k+1} \\
    \lambda_{(0^k,2^{k+1})} = \lambda_{(0^{k+1},2^k)} &= \frac{2^{2k-1}}{k+1}\cdot k \\
    \lambda_{(0^k,1,2^k)} &= \frac{2^{2k-1}}{k+1}\cdot(2k+2)
\end{align*}
and $\lambda_x = 0$ for all other representatives $x\in\Z_4^d/G_d$ yields a feasible dual auxiliary function of the radialized discrete dual Cohn-Elkies linear program in $\Z_4^d$ with $r=2$.
\end{theorem}
\begin{proof}
Clearly $\lambda_y \geq 0$ for all $y$.
We have $\mu = \mc{F}^T\lambda$, so
\[ \mu_x = \sum_{y\in\Z_4^d/G_d} \mc{F}^T_{xy}\lambda_y = 4^{-d/2} \sum_y \sum_{z\sim x} e^{-\frac{2\pi i}{4}\left<y,z\right>}\lambda_y = 2^{-2k-1} \sum_y \sum_{z\sim x} (-i)^{\left<y,z\right>}\lambda_y. \]
The only nonzero $\lambda_y$ terms are the five given in the theorem statement.
Let $x=(0^a,1^b,2^c)$ for $a+b+c=d=2k+1$.
The $y=0$ term yields $\frac{1}{4(k+1)}|\Orb(x)|$ and the $y=(2^d)$ term yields $(-1)^b\frac{1}{4(k+1)}|\Orb(x)|$.

For the $(0^k,2^{k+1})$ and $(0^{k+1},2^k)$ terms, if $b$ is odd then $\sum_i z_i$ is odd for all $z \sim x$, so using $(2^k,0^{k+1})$ instead of $(0^{k+1},2^k)$ as representative, i.e., reversing the indices of $(0^{k+1},2^k)$, we have
\[ \left<(0^k,2^{k+1}),z\right> + \left<(2^k,0^{k+1}),z\right> = 2 \sum_i z_i \equiv 2 \pmod{4}. \]
As $\lambda$ takes on the same value for these two representatives, one of these terms is congruent to 0 mod 4 and the other 2 mod 4, so the $(-i)^{\left<y,z\right>}$ terms are $\pm1$, which cancel, yielding a total of 0 for these two representatives.

If $b$ is even, then $(-i)^{\left<y,z\right>}=\pm1$ will have the same sign for these two representatives $(0^k,2^{k+1})$ and $(0^{k+1},2^k)$ for each $z\sim x$, specifically $(-1)^\ell$ where $\ell$ is the number of ones and threes of $z$ in the first $k$ indices.
So the contribution of these two representatives is
\begin{align*}
    \frac{1}{4(k+1)}\cdot k \sum_{z\sim x} (-i)^{\left<(0^k,2^{k+1}),z\right>} + (-i)^{\left<(2^k,0^{k+1}),z\right>}
    &= \frac{1}{4(k+1)}\cdot k \binom{a+c}{a}2^b\cdot 2 \sum_{\ell=0}^b (-1)^\ell \binom{k}{\ell}\binom{k+1}{b-\ell}
\end{align*}
if $b \leq k$, and
\begin{align*}
    \frac{k}{4(k+1)}& \sum_{z\sim x} (-i)^{\left<(0^k,2^{k+1}),z\right>} + (-i)^{\left<(2^k,0^{k+1}),z\right>}
    = \frac{k}{4(k+1)} \binom{a+c}{a}2^b\cdot 2 \sum_{\ell=b-k-1}^k (-1)^\ell \binom{k}{\ell}\binom{k+1}{b-\ell}
    \\ &= \frac{k}{4(k+1)} \binom{a+c}{a}2^b\cdot 2(-1)^k \sum_{\ell=0}^{2k+1-b} (-1)^\ell \binom{k}{\ell}\binom{k+1}{2k+1-b-\ell}
\end{align*}
if $b \geq k+1$.
After picking locations for the ones or threes of $z$, the $\binom{a+c}{a}$ term corresponds to picking which of the remaining locations will be zeros and the $2^b$ term corresponds to picking whether each of the $b$ indices are a one or a three.
The final sums are \emph{Krawtchouk polynomials} (with $q=2$) given by
\[ K_j(i;n) = \sum_{\ell=0}^j (-1)^\ell \binom{i}{\ell}\binom{n-i}{j-\ell}, \]
where $i,j$ are integers between 0 and $n$.
One fundamental property of the Krawtchouk polynomials (easily derived from \cite[p. 152]{macwilliams1977theory}) is that $K_{n-j}(i;n) = (-1)^i K_j(i;n)$, which allows us to simply our expression for the contribution of these two representatives to
\[ \frac{k}{4(k+1)}\binom{a+c}{a}2^{b} \cdot 2 K_b(k;2k+1), \]
for all even $b$.
Combining this with the odd $b$ case, which vanishes, yields
\[ (1+(-1)^b) \frac{k}{4(k+1)} \binom{a+c}{a}2^{b} K_b(k;2k+1). \]

The last term is the $y=(0^k,1,2^k)$ term.
If $z_{k+1}\in\{1,3\}$, negating this index negates $(-i)^{\left<y,z\right>}$ while remaining within $\Orb(x)$, so these terms cancel.
Hence, we assume either a 0 or 2 is at index $k+1$, where the 1 of $y$ is situated.
By similar reasoning to the previous representatives, the cases where $z_{k+1}=0$ collectively yield a coefficient of $K_b(k;2k)$ for $\lambda_y$ times $\binom{a+c-1}{c}2^b$ for locations of zeros vs. twos and signs of ones vs. threes, so in aggregate $\binom{a+c-1}{c}2^b K_b(k;2k)$.
Similarly, the cases where $z_{k+1}=2$ collectively yield $-\binom{a+c-1}{a}2^b K_b(k;2k)$.
As $\lambda_y = \frac{2^{2k-1}}{k+1}(2k+2)$, combining all five representatives' contributions yields
\begin{align*}
    \mu_x = \frac{2^b}{4(k+1)}\Bigg[&(1+(-1)^b)\binom{a+c}{a}\left(\binom{2k+1}{b}+kK_b(k;2k+1)\right) 
    \\ &+ (2k+2)\left(\binom{a+c-1}{c}-\binom{a+c-1}{a}\right)K_b(k;2k) \Bigg],
\end{align*}
where $|\Orb(x)|=\binom{2k+1}{b}\binom{a+c}{a}2^b$.
Another fundamental property of the Krawtchouk polynomials (see \cite[p. 152]{macwilliams1977theory}) is that $K_j(n-i;n)=(-1)^j K_j(i;n)$, which implies $K_b(k;2k)=0$ for odd $b$, and thus $\mu_x = 0$ when $b$ is odd.

As $K_0(i;n)=1$, we find that $\mu_0 = 1 = (r/2)^d$, as desired.
We also need $\mu_x = 0$ for all $0 < |x| < r=2$.
The only such $x$ are those with $c=0$ and $1 \leq b \leq 3$.
When $b=1$ or 3, our previous observation implies $\mu_x = 0$.
When $b=2$, using the fact that $K_1(i;n) = n-2i$ and $K_2(i;n) = \binom{n}{2}-2ni+2i^2$, we find $\mu_{(0^{2k-1},1^2)}=0$.

For general $x=(0^a,1^b,2^c)$, we must show $\mu_x \geq 0$.
We may assume $b$ is even as otherwise $\mu_x=0$.
As
\[ \binom{a+c-1}{c}-\binom{a+c-1}{a} = \binom{a+c}{a}\frac{a-c}{a+c}, \]
we will equivalently show
\[ \binom{2k+1}{b}+kK_b(k;2k+1) + (k+1)\frac{a-c}{a+c}K_b(k;2k) \geq 0. \]
To show this for all $a,c$, where $-1\leq \frac{a-c}{a+c}\leq 1$, this is equivalent to showing
\[ (k+1)|K_b(k;2k)| \leq \binom{2k+1}{b} + kK_b(k;2k+1). \]
For $b=0$, this yields $k+1 \leq k+1$, so we now consider $b \geq 2$.
As $K_b(k;2k+1)=K_b(k;2k)+K_{b-1}(k;2k) = K_b(k;2k)$, where the first equality follows from \cite[Lemma 2.6]{li2022unique}, we must show that
\[ \binom{2k+1}{b} \geq (k+1)|K_b(k;2k)|-kK_b(k;2k). \]
We will prove the stronger statement
\[ \binom{2k+1}{b} \geq (2k+1)|K_b(k;2k)| \iff \binom{2k}{b} \geq (2k+1-b)|K_b(k;2k)|. \]
Notice that $K_b(k;2k) = \sum_{\ell=0}^b (-1)^\ell \binom{k}{\ell}\binom{k}{b-\ell}$ is the coefficient of $x^b$ in $(1-x)^k(1+x)^k$; as $(1-x)^k(1+x)^k = (1-x^2)^k$, we find this implies
\[ K_b(k;2k) = (-1)^{b/2} \binom{k}{b/2}. \]
Thus it suffices to prove
\[ \binom{2k}{b} \geq (2k+1-b)\binom{k}{b/2} \]
for even $2 \leq b \leq 2k$.
We directly find $b=2$ and $b=2k$ yield sharp inequalities, and for $b=2k-2$ the inequality is equivalent to $k \geq 2$, so it holds.
For $4 \leq b \leq 2k-4$, notice that $\frac{k(k-1)}{2}\geq 2k-3$ holds for all integers $k \geq 2$, from which we have
\[ \binom{2k}{b} = \sum_{\ell=0}^b \binom{k}{\ell}\binom{k}{b-\ell} \geq \binom{k}{b/2}^2 \geq \binom{k}{2}\binom{k}{b/2} \geq (2k-3)\binom{k}{b/2} \geq (2k+1-b)\binom{k}{b/2}. \]
This completes the proof.
\end{proof}
\begin{corollary}\label{corollary: general dual bound odd dimension}
For odd $d \geq 5$, the $d$-dimensional Cohn-Elkies linear programming bound is greater than or equal to $\frac{1}{2(d+1)}$.
\end{corollary}
\begin{proof}
This follows directly from \cref{theorem: general dual auxiliary odd dimension}, where the objective function is
\[ m^{-d/2}\lambda_0 = 2^{-2k-1} \frac{2^{2k-1}}{k+1} = \frac{1}{4(k+1)}=\frac{1}{2(d+1)}, \]
where $d=2k+1$.
\end{proof}
This generalizes the dual bounds found in dimensions 9 and 11 in \cref{theorem: d=9,theorem: d=11} to all odd dimensions $d \geq 5$.
Unfortunately, this general dual bound decays faster than the upper center density of the best known sphere packings in each dimension, so the only new dimension for which it exceeds the best known packing is $d=13$.
For $d=13$, the best known sphere packing \cite{leech1971sphere} has upper center density $3^2\cdot2^{-8}=\frac{9}{256}=0.03515625$.
\begin{corollary}\label{corollary: d=13}
The 13-dimensional Cohn-Elkies linear programming bound is greater than or equal to $\frac{1}{28}> 0.03571428$.
\end{corollary}
The best known upper bound on $\delta_{13}$ is 0.06609354, rounded up \cite{cohn2022three}.

\section{Conclusion and open problems}\label{section: conclusion}
As the discrete Cohn-Elkies linear program, as well as its dual, both optimize over real-valued functions on $\Z_m^d$, larger values of $d$ become increasingly hard to optimize.
However, the $G_d$-invariance slightly mollifies these effects, making the radialized discrete dual noticeably more computationally tractable.
In particular, the maximum orbit size of $G_d$ is $2^d d!$, achieved by $x\in \Z_m^d$ if $x_1,\dots,x_d$ and their negatives all have distinct values.
This increases in $d$, and moreover, as $m$ increases, we expect orbits to be larger due to the lower likelihood of overlaps in the values of $\pm x_1, \dots, \pm x_d$.
Thus the radialized discrete dual optimizes over real variables indexed by $\Z_m^d/G_d$, where
\[ \frac{m^d}{2^d d!} \leq |\Z_m^d/G_d| \leq m^d, \]
and by our reasoning we expect $|\Z_m^d/G_d|$ to tend towards its lower bound as $m$ increases.
In fact, by choosing representatives $x\in\Z_m^d/G_d$ of the form $0 \leq x_1 \leq \cdots \leq x_d \leq \floor{\frac{m}{2}}$, we find $|\Z_m^d/G_d|$ is the number of ways to partition $d$ elements into $\floor{\frac{m}{2}}+1$ distinguishable sets, which yields
\[ |\Z_m^d/G_d| = \binom{d+\floor{\frac{m}{2}}}{d} \]
where as $m$ increases we see this tends towards $\frac{m^d}{2^d d!}$, i.e., $|\Z_m^d/G_d|\sim \frac{m^d}{2^d d!}$ as a function of $m$.
In addition, the smallest value of $m$ needed to beat the best known packing empirically decreases with $d$: for $d\in\{3,4,5,6,7,9,10,11,12,13\}$, these values are 21, 16, 10, 8, 8, 4, 4, 4, 4, and 4, respectively.
Thus, the radialized discrete dual Cohn-Elkies linear program is more tractable than one may na\"ively expect.

While our discussion in \cref{section: Discrete LP} provides an intuitive argument as to why the discrete Cohn-Elkies linear programming bound should converge to the Cohn-Elkies linear programming bound as $m,r\to\infty$, a rigorous proof of this is left as an open problem.
It would also be of interest to know the corresponding rate of convergence.

One potential avenue for further research is lifting the discrete dual solutions from $\Z_m^d$ back to dual solutions on $\R^d$.
Recall that the discrete dual solutions for this paper are available as ancillary files on the arXiv.org e-print archive.
As the discrete reduction incorporates some loss into this process of creating dual bounds, understanding the preimages of discrete dual solutions may enable a more directed search for dual solutions of the Cohn-Elkies linear program.
Directly working with the dual in $\R^d$ avoids this loss issue inherent to the discretization and has the potential of creating better dual bounds.
However, the dual has previously been essentially intractable; perhaps the preimages of discrete dual solutions may motivate an optimization over certain classes of objects in the dual feasible space that yield near-optimal dual bounds. 

\section*{Acknowledgments}
We sincerely thank Henry Cohn for his instrumental mentorship and support throughout the research process as well as editing feedback.
This research was done, in part, at the University of Minnesota Duluth REU in 2022 with support from Jane Street Capital, for which we thank Prof.\@ Joe Gallian for providing this wonderful opportunity.

\bibliographystyle{amsinit}
\bibliography{ref}

\end{document}